\documentclass{amsart}
\usepackage{amsfonts,amssymb,amsmath,amsthm}
\usepackage{url}
\usepackage{enumerate}
\usepackage{hyperref}
\usepackage{stackengine}
\usepackage{scalerel}
\newtheorem{theorem}{Theorem}[section]

\newtheorem{corollary}[theorem]{Corollary}
\theoremstyle{definition}
\newtheorem{definition}[theorem]{Definition}
\newtheorem{remark}[theorem]{Remark}
\newtheorem{example}[theorem]{Example}
\numberwithin{equation}{section}

\begin{document}
\title[Biharmonic and Interpolating Sesqui-Harmonic Vector Fields\ldots]{Biharmonic and Interpolating Sesqui-Harmonic Vector Fields with Respect to the $\varphi$-Sasakian Metric}
	
\author[A. Zagane]{Abderrahim Zagane}
\address{Ahmed Zabana University-Relizane, Faculty of Science and Technology, Department of Mathematics, 48000, Relizane, Algeria}
\email{Zaganeabr2018@gmail.com}

\author[K. Biroud]{Kheireddine Biroud}
\address{Kheireddine Biroud \\ Higher School of Management-Tlemcen \\  13000,Tlemcen, Algeria}
\email{kh\_biroud@yahoo.fr}

\author[M. Djilali]{Medjahed Djilali}
\address{Ahmed Zabana University-Relizane, Faculty of Science and Technology,
	Department of Mathematics, 48000, Relizane, Algeria}
\email{medjahed.djilali@univ-relizane.dz}
		
\begin{abstract}

This work investigates biharmonic and interpolating sesqui-harmonic vector fields on the tangent bundle of a para-K\"{a}hler--Norden manifold $(M^{2m}, \varphi, g)$  endowed with the $\varphi$-Sasaki metric. We derive the first variation of the bienergy and interpolating sesqui-energy functionals, restricted to the space of vector fields. Explicit characterizations are established for vector fields satisfying the corresponding variational conditions—namely, biharmonicity and interpolating sesqui-harmonicity. Furthermore, several examples are presented to illustrate the general theory and to elucidate the distinctions between harmonic, biharmonic, and interpolating sesqui-harmonic behaviors. These results extend and complement existing research on higher-order harmonicity in pseudo-Riemannian geometry.

		
\end{abstract}
	
\maketitle

\section{Introduction}
Among the pioneering works on the existence of harmonic maps between Riemannian manifolds is the seminal paper by Eells and Sampson~\cite{E.S}, which established that any smooth map from a compact Riemannian manifold into a target manifold with non-positive sectional curvature can be deformed into a harmonic map. Since then, numerous contributions have advanced the theory of harmonic maps; see, for example,~\cite{E.L1, E.L2, Jia}.

The problem of constructing harmonic maps defined by vector fields (respectively, covector fields) has also attracted considerable attention. Such maps can be viewed as mappings from a Riemannian manifold into its tangent bundle (respectively, cotangent bundle), or as compositions of vector fields (respectively, covector fields) with smooth maps. Problems of this type have been studied in the setting of tangent bundles (see~\cite{D.P, Kon, L.D.Z, Zag11}) and, more recently, in the setting of cotangent bundles (see~\cite{B.Z2, Z.G, Zag16}).

Biharmonic maps arise as a natural higher-order generalization of harmonic maps within the framework of variational geometry. They are defined as critical points of the bienergy functional. This notion was first introduced and systematically developed by G. Jiang in the mid-1980s~\cite{Jia}, who derived the associated Euler--Lagrange equations and established fundamental variational properties. Since Jiang’s pioneering work, biharmonic maps have been extensively investigated (see, for instance,~\cite{A.K.O, B.K, Bal, D.Z4, M.U}), leading to a rich theory that includes rigidity results, classification problems, and applications to submanifold geometry and geometric analysis.

Interpolating sesqui-harmonic maps were introduced and first systematically studied by Caddeo, Montaldo, and Oniciuc in the early 2010s~\cite{C.M.O1}. Motivated by the theory of biharmonic maps, the authors considered a variational functional obtained as a linear combination of the Dirichlet energy and the bienergy, thereby interpolating between harmonic and biharmonic maps; see also~\cite{ Bra1, Bra2, C.M.O2, K.A.O, K.O.D}.

In parallel with these developments, the geometry of tangent bundles endowed with natural metrics has become an active area of research. Among the most studied examples are the Sasaki metric \cite{Sas} and its various generalizations \cite{A.S, M.T}, which encode both the geometry of the base manifold and the behavior of vector fields. When the base manifold is endowed with additional structures, such as almost complex or almost product structures, the induced geometry on the tangent bundle becomes even richer \cite{Zag7,Zag21,Zag26}.

In this paper, we investigate several biharmonic and interpolating sesqui-harmonic problems on the tangent bundle endowed with $\varphi$-Sasaki metrics over a para-K\"{a}hler--Norden manifold $(M^{2m},\varphi,g)$. This study is carried out by analyzing the influence of certain geometric properties of the base manifold—such as the induced Levi-Civita connection and the Riemannian curvature—on its tangent bundle, as well as their effects on the behavior of biharmonic and interpolating sesqui-harmonic vector fields with respect to the $\varphi$-Sasaki metric.

After presenting a general introduction, Section~\ref{sec:02} recalls basic definitions and fundamental results concerning the geometry of the tangent bundle over a para-K\"{a}hler--Norden manifold $(M^{2m}, \varphi, g)$, including explicit formulas for the induced Levi-Civita connection and the Riemannian curvature tensors.

Section~\ref{sec:03} establishes the basic properties of the energy and bienergy functionals, as well as several harmonicity properties of vector fields with respect to the $\varphi$-Sasaki metric.

Section~\ref{sec:04} is devoted to the study of biharmonic problems for vector fields, considered both as biharmonic maps and as biharmonic vector fields. We provide necessary and sufficient conditions for the biharmonicity of vector fields in each case, including the first variation of the bienergy functional.

Finally, Section~\ref{sec:05} investigates interpolating sesqui-harmonic problems for vector fields, viewed as interpolating sesqui-harmonic maps and as interpolating sesqui-harmonic vector fields. Necessary and sufficient conditions for interpolating sesqui-harmonicity are established in each setting, including the first variation of the interpolating sesqui-energy functional. The paper concludes with several examples illustrating harmonicity, biharmonicity, and interpolating sesqui-harmonicity on tangent bundles endowed with $\varphi$-Sasaki metrics.

This work enlarges the class of biharmonic and interpolating sesqui-harmonic problems on tangent bundles, providing new insights into the relationship between a manifold and its tangent bundle, as well as into the notions of biharmonicity and interpolating sesqui-harmonicity in various problems related to vector fields.
 
\section{Basic properties of tangent bundle with the $\varphi$-Sasaki metric}\label{sec:02}
An almost complex manifold $(M,\varphi)$ is a differentiable manifold $M$ endowed with a tensor field of type $(1,1)$ satisfying
$$\varphi^{2}=I, \quad \varphi \neq \pm I,$$
such that the two eigenbundles $TM^{+}$ and $TM^{-}$ associated with the eigenvalues $+1$ and $-1$ of $\varphi$, respectively, are required to have the same rank. Consequently, the dimension of an almost para-complex manifold is necessarily even.

An almost para-complex Norden manifold $(M^{2m}, \varphi, g)$ is an almost complex manifold $(M,\varphi)$ endowed with a Riemannian metric $g$ such that
\begin{equation*}
g(\varphi V, W) = g(V, \varphi W),
\end{equation*}
for all vector fields $V$ and $W$ on $M$. In this case, the metric $g$ is called a para-Norden metric (or $B$-metric) \cite{S.I.E}.

A para-K\"{a}hler--Norden manifold is an almost para-complex Norden manifold $(M^{2m}, \varphi, g)$ for which
$$\nabla \varphi = 0,$$
where $\nabla$ denotes the Levi-Civita connection associated with the metric $g$ \cite{S.G.I, S.I.E}.

On a para-K\"{a}hler--Norden manifold $(M^{2m}, \varphi, g)$ \cite{S.I.E}, the Riemann curvature tensor is pure; that is, the following relations hold:
\begin{equation}\label{cur-pure}
\mathrm{R}(\varphi V, W) = \mathrm{R}(V, \varphi W) = \mathrm{R}(V, W)\varphi = \varphi \mathrm{R}(V, W),
\end{equation}
for all vector fields $V$ and $W$ on $M$.

Now let $(M,g)$ be a smooth Riemannian manifold with Levi--Civita connection $\nabla$. We denote by $TM$ its tangent bundle and by $\pi \colon TM \to M$ the canonical projection.

For each point $(p,v) \in TM$, the tangent space $T_{(p,v)}TM$ admits a decomposition into horizontal and vertical distributions, denoted respectively by $\mathcal{H}_{(p,v)}$ and $\mathcal{V}_{(p,v)}$, with respect to the Levi--Civita connection $\nabla$, namely,
$$T_{(p,v)}TM = \mathcal{H}_{(p,v)} \oplus \mathcal{V}_{(p,v)}.$$

For a vector field $W$ on $M$, we denote by $W^{H}$ and $W^{V}$ its horizontal and vertical lifts to $TM$, respectively. These lifts satisfy the following properties:
$$d\pi(W^{H}) = W, \quad d\pi(W^{V}) = 0.$$

\begin{definition}\cite{Zag7,Zag21}
Given a para-K\"{a}hler--Norden manifold $(M^{2m}, \varphi, g)$. On its tangent bundle $TM$, we define the $\varphi$-Sasaki metric, denoted by $g^{\varphi}$, as follows:
\begin{eqnarray*}
 g^{\varphi}({}^{H}\!W,{}^{H}\!Z) &=&g(W,Z),\\
 g^{\varphi}({}^{H}\!W,{}^{V}\!Z) &=& 0,\\
 g^{\varphi}({}^{V}\!W,{}^{V}\!Z) &=& g(W,\varphi Z),
\end{eqnarray*}
for all vector fields $W$ and $Z$ on $M$.
\end{definition}

The Levi-Civita connection $\widetilde{\nabla}$ of the tangent bundle $TM$ endowed with the $\varphi$-Sasaki metric $g^{\varphi}$ is described in~\cite{Zag7} by the following result.

\begin{theorem}\label{thm:01}
Given a para-K\"{a}hler--Norden manifold $(M^{2m}, \varphi, g)$ and its tangent bundle $(TM,g^{\varphi})$ endowed with the $\varphi$-Sasaki metric. Then, the Levi-Civita connection $\widetilde{\nabla}$ satisfies:
\begin{eqnarray*}
(1)\quad (\widetilde{\nabla}_{{}^{H}\!W}{}^{H}\!Z)&=&{}^{H}\!(\nabla_{W}Z)-\frac{1}{2}{}^{V}\!(\mathrm{R}(W,Z)v),\\
(2)\quad (\widetilde{\nabla}_{{}^{H}\!W}{}^{V}\!Z)&=&{}^{V}\!(\nabla_{W}Z)+\dfrac{1}{2}{}^{H}\!(\mathrm{R}(\varphi v,Z)W),\\
(3)\quad (\widetilde{\nabla}_{{}^{V}\!W}{}^{H}\!Z)&=&\dfrac{1}{2}{}^{H}\!(\mathrm{R}(\varphi v,W)Z),\\
(4)\quad (\widetilde{\nabla}_{{}^{V}\!W}{}^{V}\!Z)&=&0,
\end{eqnarray*}
for all vector fields $W$ and $Z$ on $M$, where $\nabla$ is the Levi-Civita connection of $(M^{2m}, \varphi, g)$ and $\mathrm{R}$ is its curvature tensor.
\end{theorem}

The Riemann curvature tensor $\widetilde{\mathrm{R}}$ of tangent bundle $(TM,g^{\varphi})$ is given in \cite{Zag21} by the following theorem.

\begin{theorem}\label{thm:02}
Given a para-K\"{a}hler--Norden manifold $(M^{2m}, \varphi, g)$ and its tangent bundle $(TM,g^{\varphi})$ endowed with the $\varphi$-Sasaki metric. Then, the following formulas hold:
\begin{eqnarray*}
\widetilde{\mathrm{R}}({}^{H}\!X,{}^{H}\!Y){}^{H}\!Z &=&{}^{H}\!(\mathrm{R}(X,Y)Z)+\dfrac{1}{2}{}^{H}\!(\mathrm{R}(\varphi v,\mathrm{R}(X,Y)v)Z))\notag\\
&&+\dfrac{1}{4}{}^{H}\!(\mathrm{R}(\varphi v,\mathrm{R}(X,Z)v)Y)-\dfrac{1}{4}{}^{H}\!(\mathrm{R}(\varphi v,\mathrm{R}(Y,Z)v)X)\\
&&+\dfrac{1}{2}{}^{V}\!((\nabla_{Z}R)(X,Y)v),\\
\widetilde{\mathrm{R}}({}^{H}\!X,{}^{V}\!Y){}^{V}\!Z&=&-\dfrac{1}{2}{}^{H}\!(\mathrm{R}(\varphi Y,Z)X)-\dfrac{1}{4}{}^{H}\!(\mathrm{R}(v,Y)\mathrm{R}(v,Z)X),\\
\widetilde{\mathrm{R}}({}^{V}\!X,{}^{V}\!Y){}^{H}\!Z&=&\dfrac{1}{4}{}^{H}\!(\mathrm{R}(v,X)\mathrm{R}(v,Y)Z)-\dfrac{1}{4}{}^{H}\!(\mathrm{R}(v,Y)\mathrm{R}(v,X)Z))\\
&&+{}^{H}\!\big(\mathrm{R}(\varphi X,Y)Z),\\
\widetilde{\mathrm{R}}({}^{H}\!X,{}^{V}\!Y){}^{H}\!Z&=&\dfrac{1}{4}{}^{V}\!(\mathrm{R}(\mathrm{R}(\varphi v,Y)Z,X)v)+\dfrac{1}{2}{}^{H}\!((\nabla_{X}R)(\varphi v,Y)Z)\\
&&+\dfrac{1}{2}{}^{V}\!(\mathrm{R}(X,Z)Y),\\
\widetilde{\mathrm{R}}({}^{H}\!X,{}^{H}\!Y){}^{V}\!Z&=&\dfrac{1}{2}{}^{H}\!((\nabla_{X}R)(\varphi v,Z)Y)-\dfrac{1}{2}{}^{H}\!((\nabla_{Y}R)(\varphi v,Z)X))\notag\\
&&+\dfrac{1}{4}{}^{V}\!(\mathrm{R}(\mathrm{R}(\varphi u,Z)Y,X)v)-\dfrac{1}{4}{}^{V}\!(\mathrm{R}(\mathrm{R}(\varphi v,Z)X,Y)v)\\
&&+{}^{V}\!(\mathrm{R}(X,Y)Z),\\
\widetilde{\mathrm{R}}({}^{V}\!X,{}^{V}\!Y){}^{V}\!Z&=&0,
\end{eqnarray*}	
for all vector fields $X,Y$ and $Z$  on $M$.
\end{theorem}

\section{Basic properties of Energy and bienergy functional}\label{sec:03}

Let $\phi:(M,g)\to (N,h)$ be a smooth map between two Riemannian manifolds. The map $\phi$ is said to be harmonic if it is a critical point of the energy functional
\begin{equation*}
\mathrm{E}(\phi)=\frac{1}{2}\int_{K}|d\phi|^{2}\,v_{g},
\end{equation*}
for any compact domain $K\subseteq M$. Here, $|d\phi|$ denotes the Hilbert--Schmidt norm of $d\phi$, and $v_{g}$ is the Riemannian volume form on $M$. Equivalently, $\phi$ is harmonic if it satisfies the associated Euler--Lagrange equation
\begin{equation*}
\tau(\phi)=\mathrm{Tr}_{g}\nabla d\phi=0,
\end{equation*}
where $\tau(\phi)$ is the tension field of $\phi$.

A smooth map $\phi$ is said to be biharmonic if it is a critical point of the bienergy functional
\begin{equation}\label{eq:01}
\mathrm{E}_{2}(\phi)=\frac{1}{2}\int_{K}|\tau(\phi)|^{2}\,v_{g}.
\end{equation}
Equivalently, $\phi$ is biharmonic if it satisfies the corresponding Euler--Lagrange equation
\begin{eqnarray}\label{eq:02}
\tau_{2}(\phi)=\Delta^{\phi}\tau(\phi)-\mathrm{Tr}_{g}\mathrm{R}^{N}(\tau(\phi),d\phi)d\phi=0,
\end{eqnarray}
where $\tau_{2}(\phi)$ denotes the bitension field of $\phi$, $\Delta^{\phi}$ is the rough Laplacian acting on sections of the pull-back bundle $\phi^{-1}TN$, defined by
\begin{eqnarray}\label{eq:03}
\Delta^{\phi}\tau(\phi)=-\mathrm{Tr}_{g}\left(\nabla^{\phi}_{\ast}\nabla^{\phi}_{\ast}-\nabla^{\phi}_{\nabla_{\ast}\ast}\right)\tau(\phi),
\end{eqnarray}
and $\mathrm{R}^{N}$ is the curvature tensor of the target manifold $N$. Clearly, every harmonic map is biharmonic; therefore, it is of interest to construct and study proper biharmonic maps, that is, biharmonic maps which are not harmonic.

The interpolating sesqui-energy functional is defined as a linear combination of the energy and the bienergy functionals:
\begin{equation}\label{eq:04}
\mathrm{E}_{\delta_{1},\delta_{2}}(\phi)=\delta_{1}\int_{K}|d\phi|^{2}\,v_{g}+\delta_{2}\int_{K}|\tau(\phi)|^{2}\,v_{g}=2\delta_{1}\mathrm{E}(\phi)+2\delta_{2}\mathrm{E}_{2}(\phi),
\end{equation}
where $\delta_{1},\delta_{2}\in\mathbb{R}$ see\cite{Bra1}. A map $\phi$ is called interpolating sesqui-harmonic if it is a critical point of the functional \eqref{eq:04}. Equivalently, $\phi$ satisfies the Euler--Lagrange equation \cite{K.A.O}
\begin{eqnarray}\label{eq:05}
\tau_{\delta_{1},\delta_{2}}(\phi)=\delta_{1}\tau(\phi)+\delta_{2}\tau_{2}(\phi)=0,
\end{eqnarray}
where $\tau_{\delta_{1},\delta_{2}}(\phi)$ is referred to as the interpolating sesqui-tension field. Clearly, every harmonic map is interpolating sesqui-harmonic. Consequently, the study of proper interpolating sesqui-harmonic maps is of particular interest.\quad

Below, we present several harmonic properties of the tangent bundle endowed with $\varphi$-Sasakian metrics over a para-K\"{a}hler--Norden manifold $(M^{2m}, \varphi, g)$. We begin with the following fundamental property:
\begin{eqnarray}\label{eq:06}
d_{p}Y(X_{p}) = {}^{H}\!X_{(p,v)} + {}^{V}\!(\nabla_X Y)_{(p,v)},
\end{eqnarray}
for all vector fields $X$ and $Y$ on $M$ and for any $(p,v) \in TM$ such that $Y_{p} = v$ (see \cite{Kon, L.D.Z}).

A vector field $\xi$ on $(M^{2m}, g, \varphi)$ can be viewed as an immersion
$$
\xi : (M^{2m}, g, \varphi) \to (TM, g^{\varphi}),
$$
into its tangent bundle $TM$ endowed with the $\varphi$-Sasaki metric $g^{\varphi}$.

The energy $\mathrm{E}(\xi)$ of the vector field $\xi$ is defined as the energy of the corresponding map
$$\xi : (M^{2m}, \varphi, g) \to (TM, g^{\varphi}).$$

If $(M^{2m}, \varphi, g)$ is a compact para-K\"{a}hler--Norden manifold, then $\mathrm{E}(\xi)$ is given in  \cite{Zag26} by
\begin{equation*}
\mathrm{E}(\xi) = m\,\mathrm{Vol}(M) + \frac{1}{2}\int_M \mathrm{Tr}_{g} \, g(\nabla_{\ast}\xi, \varphi \nabla_{\ast}\xi)\, v_{g}.
\end{equation*}

The tension field $\tau(\xi)$ of $\xi$ is expressed in \cite{Zag26} as
\begin{equation}\label{eq:08}
\tau(\xi) = {}^{H}\!S(\xi) - {}^{V}\!\bar{\Delta}\xi,
\end{equation}
where
\begin{equation}\label{eq:09}
\bar{\Delta}\xi = -\mathrm{Tr}_{g} \nabla^{2} \xi= -\mathrm{Tr}_{g} \bigl(\nabla_{\ast}\nabla_{\ast} - \nabla_{\nabla_{\ast}\ast}\bigr)\xi
\end{equation}
is the rough Laplacian of $\xi$, and
\begin{equation}\label{eq:10}
S(\xi) = \mathrm{Tr}_{g} \mathrm{R}(\varphi \xi, \nabla_{\ast} \xi)\ast.
\end{equation}

Moreover, for any smooth function $f$ on $M$, the following formulas hold (see \cite{Zag26}):
\begin{equation}\label{eq:11}
\bar{\Delta}(f\xi) = f\,\bar{\Delta}\xi - (\Delta f)\xi - 2 \nabla_{\mathrm{grad}\, f} \xi,
\end{equation}
where $\Delta f$ denotes the Laplace operator of $f$ and $\mathrm{grad}\, f$ its gradient, and
\begin{equation}\label{eq:12}
S(f\xi) = f^{2} S(\xi).
\end{equation}

A vector field $\xi$ on $M$ is a harmonic map from $(M^{2m}, g, \varphi)$ to $(TM, g^{\varphi})$ if and only if
\begin{equation}\label{eq:13}
\tau(\xi) = 0\Leftrightarrow S(\xi) = 0 \quad \text{and} \quad \bar{\Delta}\xi = 0.
\end{equation}
The vector field $\xi$ is said to be a harmonic vector field if it is a critical point of the energy functional $\mathrm{E}(\xi)$ with respect to variations through vector fields, which is equivalent to $\bar{\Delta}\xi = 0.$

Moreover, $\xi$ is a harmonic map if and only if $\xi$ is a harmonic vector field and satisfies $S(\xi) = 0$ (see \cite{Zag26}).

\section{Biharmonicity of vector field $\xi:(M^{2m},\varphi,g)\to (TM,g^{\varphi})$} \label{sec:04}
\begin{theorem}
Given a para-K\"{a}hler--Norden manifold $(M^{2m}, \varphi, g)$ and its tangent bundle $(TM,g^{\varphi})$ endowed with the $\varphi$-Sasaki metric. For a vector field $\xi$ on $M$, the bi-tension field associated with the map $\xi:(M^{2m},\varphi,g)\to (TM,g^{\varphi})$ is given by: 
\begin{eqnarray}\label{eq:14}
\tau_{2}(\xi) &=&{}^{H}\!\big[\bar{\Delta}S(\xi)+\mathrm{R}(\varphi \xi,\bar{\Delta}\xi)S(\xi)-\mathrm{Tr}_{g}\big((\nabla_{S(\xi)}R)(\varphi \xi,\nabla_{\ast}\xi)\ast\nonumber\\
&&+\mathrm{R}(\varphi \xi,\nabla_{\ast}\xi)\nabla_{\ast}S(\xi)-\mathrm{R}(\varphi \xi,\mathrm{R}(\ast,S(\xi)) \xi)\ast+\mathrm{R}(S(\xi),\ast)\ast\nonumber\\
&&-\mathrm{R}(\varphi \xi,\nabla_{\ast}\bar{\Delta}\xi)\ast-\mathrm{R}(\varphi\bar{\Delta}\xi,\nabla_{\ast}\xi)\ast\big)\big]\nonumber\\
&&+{}^{V}\!\big[-\bar{\Delta}^{2}\xi+\mathrm{Tr}_{g}\big((\nabla_{\ast}R)(\ast,S(\xi))\xi+ \mathrm{R}(\ast,\nabla_{\ast}S(\xi))\xi\nonumber\\
&&+2\mathrm{R}(\ast,S(\xi))\nabla_{\ast}\xi\big)\big],
\end{eqnarray}
where $\bar{\Delta}^{2}\xi=\bar{\Delta}\bar{\Delta}\xi$.
\end{theorem}

\begin{proof}
Let $(p,v)\in TM$, and $\{e_{i}\}_{i=\overline{1,2m}}$ be a local orthonormal frame on $(M^{2m},\varphi,g)$ such that $\nabla_{e_{i}}e_{i}=0$ at the point $p$ and $\xi_{p}=v$. From~$\eqref{eq:02}$ and~$\eqref{eq:03}$, we have 
\begin{eqnarray}\label{eq:15}
\tau_{2}(\xi)=\Delta^{\xi}\tau(\xi)-\mathrm{Tr}_{g}\widetilde{\mathrm{R}}(\tau(\xi),d\xi)d\xi,
\end{eqnarray}
To compute this expression, we apply Theorems~\ref{thm:01} and~\ref{thm:02}, together with formulas~\eqref{cur-pure}, \eqref{eq:06}, and~\eqref{eq:08}. A direct computation yields
\begin{eqnarray*}
\nabla^{\xi}_{e_{i}}\tau(\xi)&=&\widetilde{\nabla}_{d\xi (e_{i})}\tau(\xi)\\
&=&{}^{H}\!\big[\nabla_{e_{i}}S(\xi)-\frac{1}{2}\mathrm{R}(\varphi \xi,\bar{\Delta}\xi)e_{i}+\frac{1}{2}\mathrm{R}(\varphi \xi,\nabla_{e_{i}}\xi)S(\xi)\big]\\
&&-{}^{V}\!\big[\nabla_{e_{i}}\bar{\Delta}\xi+\frac{1}{2} \mathrm{R}(e_{i},S(\xi))\xi\big],
\end{eqnarray*}
and consequently
\begin{eqnarray}\label{eq:16}
\Delta^{\xi}\tau(\xi)&=&-\sum_{i=1}^{m}\nabla^{\xi}_{e_{i}}\nabla^{\xi}_{e_{i}}\tau(\xi)\nonumber\\
&=&-\sum_{i=1}^{m}\widetilde{\nabla}_{d\xi (e_{i})}\widetilde{\nabla}_{d\xi (e_{i})}\tau(\xi)\nonumber\\ 
&=&{}^{H}\!\big[\bar{\Delta}S(\xi)+\mathrm{Tr}_{g}\big(\frac{1}{2}\nabla_{\ast}(\mathrm{R}(\varphi \xi,\bar{\Delta}\xi)\ast)-\frac{1}{2}\nabla_{\ast}(\mathrm{R}(\varphi \xi,\nabla_{\ast}\xi)S(\xi))\nonumber\\
&&-\frac{1}{2}\mathrm{R}(\varphi \xi,\nabla_{\ast}\xi)\nabla_{\ast}S(\xi)+\frac{1}{2}(\mathrm{R}(\varphi \xi,\nabla_{\ast}\bar{\Delta}\xi)\ast\nonumber\\
&&+\frac{1}{4}\mathrm{R}(\varphi \xi,\mathrm{R}(\ast,S(\xi))\xi)\ast+\frac{1}{4}\mathrm{R}(\varphi \xi,\nabla_{\ast}\xi)\mathrm{R}(\varphi \xi,\bar{\Delta}\xi)\ast\nonumber\\
&&-\frac{1}{4}\mathrm{R}(\xi,\nabla_{\ast}\xi)\mathrm{R}(\xi,\nabla_{\ast}\xi)S(\xi)\big)\big]\nonumber\\
&&+{}^{V}\!\big[-\bar{\Delta}^{2}\xi+\mathrm{Tr}_{g}\big(\frac{1}{2}\nabla_{\ast}(\mathrm{R}(\ast,S(\xi))\xi)+\frac{1}{2} \mathrm{R}(\ast,\nabla_{\ast}S(\xi))\xi\nonumber\\
&&+\frac{1}{4} \mathrm{R}(\ast,\mathrm{R}(\varphi \xi,\nabla_{\ast}\xi)S(\xi))\xi-\frac{1}{4}\mathrm{R}(\ast,\mathrm{R}(\varphi \xi,\bar{\Delta}\xi)\ast)\xi\big)\big].
\end{eqnarray}
Similarly, we obtain	
\begin{eqnarray}\label{eq:17}
-\mathrm{Tr}_{g}\widetilde{\mathrm{R}}(\tau(\xi),d\xi)d\xi&=&-\sum_{i=1}^{m}\widetilde{\mathrm{R}}(\tau(\xi),d\xi(e_{i}))d\xi(e_{i})\nonumber\\
&=&{}^{H}\!\big[\frac{1}{2}\mathrm{R}(\varphi \xi,\bar{\Delta}\xi)S(\xi)+\mathrm{Tr}_{g}\big(-\mathrm{R}(S(\xi),\ast)\ast\nonumber\\
&&-\frac{3}{4}\mathrm{R}(\varphi \xi,\mathrm{R}(S(\xi),\ast)\xi)\ast-(\nabla_{S(\xi)}R)(\varphi \xi,\nabla_{\ast}\xi)\ast\nonumber\\
&&-\frac{1}{2}(\nabla_{\ast}R)(\varphi \xi,\bar{\Delta}\xi)\ast+\frac{1}{4}\mathrm{R}(\xi,\nabla_{\ast}\xi)\mathrm{R}(\xi,\bar{\Delta}\xi)\ast\nonumber\\
&&+\frac{1}{2}(\nabla_{\ast}R)(\varphi \xi,\nabla_{\ast}\xi)S(\xi)+\frac{3}{2}\mathrm{R}(\varphi\bar{\Delta}\xi,\nabla_{\ast}\xi)\ast\nonumber\\
&&+\frac{1}{4}\mathrm{R}(\xi,\nabla_{\ast}\xi)\mathrm{R}(\xi,\nabla_{\ast}\xi)S(\xi)\big)\big]\nonumber\\
&&+{}^{V}\!\big[\mathrm{Tr}_{g}\big(-\frac{1}{2}(\nabla_{\ast}R)(S(\xi),\ast)\xi-\frac{3}{2} \mathrm{R}(S(\xi),\ast)\nabla_{\ast}\xi\nonumber\\
&&+\frac{1}{4}\mathrm{R}(\mathrm{R}(\varphi \xi,\nabla_{\ast}\xi)S(\xi),\ast)\xi-\frac{1}{4}\mathrm{R}(\mathrm{R}(\varphi \xi,\bar{\Delta}\xi)\ast,\ast)\xi\big)\big].
\end{eqnarray}
Finally, substituting~\eqref{eq:16} and~\eqref{eq:17} into~\eqref{eq:15} and applying the second Bianchi identity, we obtain the desired expression~\eqref{eq:14}.
\end{proof}

\begin{theorem}\label{thm:03}
Given a para-K\"{a}hler--Norden manifold $(M^{2m}, \varphi, g)$ and its tangent bundle $(TM,g^{\varphi})$ endowed with the $\varphi$-Sasaki metric. For a vector field $\xi$ on $M$. Then, $\xi:(M^{2m},\varphi,g)\to (TM,g^{\varphi})$ is biharmonic map if and only if the following two conditions are satisfied 
\begin{align*}
\bar{\Delta}S(\xi)&+\mathrm{R}(\varphi \xi,\bar{\Delta}\xi)S(\xi)-\mathrm{Tr}_{g}\big(\mathrm{R}(S(\xi),\ast)\ast+(\nabla_{S(\xi)}R)(\varphi \xi,\nabla_{\ast}\xi)\ast\nonumber\\
&+\mathrm{R}(\varphi \xi,\nabla_{\ast}\xi)\nabla_{\ast}S(\xi)-\mathrm{R}(\varphi \xi,\mathrm{R}(\ast,S(\xi)) \xi)\ast-\mathrm{R}(\varphi \xi,\nabla_{\ast}\bar{\Delta}\xi)\ast\nonumber\\
&-\mathrm{R}(\varphi\bar{\Delta}\xi,\nabla_{\ast}\xi)\ast\big)=0,
\end{align*}
and
\begin{align*}
\bar{\Delta}^{2}\xi-\mathrm{Tr}_{g}\big((\nabla_{\ast}R)(\ast,S(\xi))\xi+ \mathrm{R}(\ast,\nabla_{\ast}S(\xi))\xi+2\mathrm{R}(\ast,S(\xi))\nabla_{\ast}\xi\big)=0. 
\end{align*}
\end{theorem}

As an immediate consequence of~\eqref{eq:13} and Theorem~\ref{thm:03}, we obtain the following result.

\begin{corollary}
Given a para-K\"{a}hler--Norden manifold $(M^{2m}, \varphi, g)$ and its tangent bundle $(TM,g^{\varphi})$ endowed with the $\varphi$-Sasaki metric. For a vector field $\xi$ on $M$. If $\xi$ is a harmonic map then $\xi$ also is a biharmonic map.
\end{corollary}

\begin{corollary}
Given a para-K\"{a}hler--Norden manifold $(M^{2m}, \varphi, g)$ and its tangent bundle $(TM,g^{\varphi})$ endowed with the $\varphi$-Sasaki metric. Then every parallel vector field on $M$ is a biharmonic map.
\end{corollary}

\begin{definition}
Given a compact para-K\"{a}hler--Norden manifold $(M^{2m},\varphi,g)$ and its tangent bundle $(TM,g^{\varphi})$ endowed with the $\varphi$-Sasaki metric. For a vector field $\xi$ on $M$, the bienergy $\mathrm{E}_{2}(\xi)$ is defined as the bienergy of the corresponding map $\xi:(M^{2m},\varphi,g) \to (TM,g^{\varphi})$.
\end{definition}

\begin{definition}
Given a para-K\"{a}hler--Norden manifold $(M^{2m}, \varphi, g)$ and its tangent bundle $(TM,g^{\varphi})$ endowed with the $\varphi$-Sasaki metric. A vector field $\xi$ on $M$ is said to be a biharmonic vector field if the corresponding map $\xi :(M^{2m},\varphi,g) \to (TM,g^{\varphi})$ is a critical point of the bienergy functional $\mathrm{E}_{2}$, when variations are restricted to maps induced by vector fields.
\end{definition}

In the following theorem, we compute the first variation of the bienergy functional restricted to the space of vector fields.

\begin{theorem}\label{thm:04}
Given a compact oriented para-K\"{a}hler--Norden manifold $(M^{2m},\varphi,g)$ and its tangent bundle $(TM,g^{\varphi})$ endowed with the $\varphi$-Sasaki metric. Let $\xi$ be a vector field on $M$, and consider the bienergy functional $\mathrm{E}_{2}$ restricted to the space of vector fields. Then, for any smooth $1$-parameter variation $\Phi:M\times(-\varepsilon,\varepsilon) \to  TM$
of $\xi$ through vector fields, the following formula holds:
\begin{eqnarray*}
\frac{d}{d t}\mathrm{E}_{2}(\xi_{t})\Big|_{t=0}&=&\int_{M}g\big(\bar{\Delta}^{2}\xi-\mathrm{Tr}_{g}\big((\nabla_{\ast}\mathrm{R})(\ast,S(\xi))\xi+ \mathrm{R}(\ast,\nabla_{\ast}S(\xi))\xi\nonumber\\
&&+2\mathrm{R}(\ast,S(\xi))\nabla_{\ast}\xi\big),\varphi W\big)v_{g},
\end{eqnarray*}
where $\Phi(p,t)=\xi_{t}(p)\in T_{p}M$ for all $p\in M$ and $|t|<\varepsilon$ $(\varepsilon>0)$. Moreover, the vector field $W$ on $M$ is defined by
\begin{eqnarray*}
W(p)= \lim_{t\to 0}\frac{1}{t}\big(\xi_{t}(p)-\xi(p)\big)= \frac{d}{dt}\Phi_{p}(t)\Big|_{t=0},
\; p\in M,
\end{eqnarray*}
where $\Phi_{p}(t)=\xi_{t}(p)$ for $(p,t)\in M\times(-\varepsilon,\varepsilon)$.
\end{theorem}

\begin{proof}
Let $\Phi:M\times(-\varepsilon,\varepsilon)\to TM$ be a smooth one-parameter variation of the vector field $\xi$, that is,
$$\Phi(p,t)=\xi_{t}(p)\in T_{p}M,
\quad (p,t)\in M\times(-\varepsilon,\varepsilon),$$
with $\Phi(p,0)=\xi_{0}(p)=\xi(p)$. 

By the general theory of biharmonic maps (see, for instance,~\cite{Jia}), we have
\begin{eqnarray}\label{eq:19}
\frac{d}{dt}\mathrm{E}_{2}(\xi_{t})\Big|_{t=0}=-\int_{M}g^{\varphi}(\tau_{2}(\xi),\mathcal{V})\,v_{g},
\end{eqnarray}
where $\mathcal{V}$ denotes the infinitesimal variation induced by $\Phi$, namely,
$$\mathcal{V}(p)
=d_{(p,0)}\Phi\!\big(0,\frac{d}{dt}\big)\Big|_{t=0}
=d\Phi_{p}\!\big(\frac{d}{dt}\big)\Big|_{t=0}
=\frac{d}{dt}\xi_{t}(p)\Big|_{t=0}
\in T_{(p,\xi(p))}TM.$$

It is well known (see~\cite[p.~58]{D.P}) that
\begin{eqnarray}\label{eq:20}
\mathcal{V}={}^{V}\!W\circ \xi,
\end{eqnarray}
where ${}^{V}\!W$ denotes the vertical lift of $W$.

Finally, taking into account formulas~\eqref{eq:14}, \eqref{eq:19}, and~\eqref{eq:20}, we obtain
\begin{eqnarray*}
\frac{d}{dt}\mathrm{E}_{2}(\xi_{t})\Big|_{t=0}&=&-\int_{M}g^{\varphi}(\tau_{2}(\xi),{}^{V}\!W)\,v_{g}\\
&=&\int_{M}g\big(\bar{\Delta}^{2}\xi-\mathrm{Tr}_{g}\big((\nabla_{\ast}R)(\ast,S(\xi))\xi
+\mathrm{R}(\ast,\nabla_{\ast}S(\xi))\xi\\
&&+2\mathrm{R}(\ast,S(\xi))\nabla_{\ast}\xi\big),\varphi W\big)\,v_{g}.
\end{eqnarray*}
\end{proof}

\begin{remark}
Theorem~\ref{thm:04} remains valid when $(M^{2m},\varphi,g)$ is a non-compact para-K\"{a}hler--Norden manifold. Indeed, in the non-compact case, one may consider an open subset $D\subset M$ with compact closure and choose a variation vector field $W$ with compact support contained in $D$. Under these assumptions, Theorem~\ref{thm:04} holds in the following form:	
\begin{eqnarray*}
\frac{d}{d t}\mathrm{E}_{2}(\xi_{t})\Big|_{t=0}&=&\int_{D}g\big(\bar{\Delta}^{2}\xi-\mathrm{Tr}_{g}\big((\nabla_{\ast}R)(\ast,S(\xi))\xi+ \mathrm{R}(\ast,\nabla_{\ast}S(\xi))\xi\nonumber\\
&&+2\mathrm{R}(\ast,S(\xi))\nabla_{\ast}\xi\big),\varphi W\big)v_{g},
\end{eqnarray*}	
\end{remark}

\begin{corollary}\label{cor:01}
Given a para-K\"{a}hler--Norden manifold $(M^{2m}, \varphi, g)$ and its tangent bundle $(TM,g^{\varphi})$ endowed with the $\varphi$-Sasaki metric. A vector field $\xi$ on $M$ is a biharmonic vector field if and only if	
\begin{eqnarray*}
\bar{\Delta}^{2}\xi=\mathrm{Tr}_{g}\big((\nabla_{\ast}R)(\ast,S(\xi))\xi+ \mathrm{R}(\ast,\nabla_{\ast}S(\xi))\xi+2\mathrm{R}(\ast,S(\xi))\nabla_{\ast}\xi\big). 
\end{eqnarray*}
\end{corollary}

From Theorem~\ref{thm:03} and Corollary~\ref{cor:01}, we obtain the following result.

\begin{corollary}
Given a para-K\"{a}hler--Norden manifold $(M^{2m}, \varphi, g)$ and its tangent bundle $(TM,g^{\varphi})$ endowed with the $\varphi$-Sasaki metric. A vector field $\xi$ on $M$ is a biharmonic map if and only if $\xi$ is biharmonic vector field and satisfies
\begin{align*}
\bar{\Delta}S(\xi)&+\mathrm{R}(\varphi \xi,\bar{\Delta}\xi)S(\xi)-\mathrm{Tr}_{g}\big(\mathrm{R}(S(\xi),\ast)\ast+(\nabla_{S(\xi)}R)(\varphi \xi,\nabla_{\ast}\xi)\ast\\
&+\mathrm{R}(\varphi \xi,\nabla_{\ast}\xi)\nabla_{\ast}S(\xi)-\mathrm{R}(\varphi \xi,\mathrm{R}(\ast,S(\xi)) \xi)\ast-\mathrm{R}(\varphi \xi,\nabla_{\ast}\bar{\Delta}\xi)\ast\\
&-\mathrm{R}(\varphi\bar{\Delta}\xi,\nabla_{\ast}\xi)\ast\big)=0,
\end{align*}
\end{corollary}

\begin{theorem}\label{thm:05}
Given a flat para-K\"{a}hler--Norden manifold $(M^{2m}, \varphi, g)$ and its tangent bundle $(TM,g^{\varphi})$ endowed with the $\varphi$-Sasaki metric. For a vector field $\xi$ on $M$, the following statements are equivalent:
\begin{itemize}
\item[(i)] $\xi$ is a biharmonic map on $M$,
\item[(ii)] $\xi$ is a biharmonic vector field on $M$,
\item[(iii)] $\bar{\Delta}^{2}\xi=0$.
\end{itemize}
\end{theorem}

\begin{corollary}\label{cor:02}
Given a para-K\"{a}hler--Norden manifold $(M^{2m}, \varphi, g)$ and its tangent bundle $(TM,g^{\varphi})$ endowed with the $\varphi$-Sasaki metric. Then every parallel vector field on $M$ is a biharmonic vector field.
\end{corollary}

\begin{theorem}\label{thm:06}
Given a compact oriented para-K\"{a}hler--Norden manifold $(M^{2m}, \varphi, g)$ and its tangent bundle $(TM,g^{\varphi})$ endowed with the $\varphi$-Sasaki metric. Then every harmonic vector field on $M$ is biharmonic vector field.
\end{theorem}

\begin{proof}
The proof follows directly from Theorem~4.4 in \cite{Zag26} and Corollary~\ref{cor:02}.
\end{proof}

\begin{remark}
In general, if $(M^{2m},\varphi,g)$ is a non-compact para-K\"{a}hler--Norden manifold, there may exist harmonic vector fields that are not biharmonic, and vice versa. See Example~\ref{exa:01},~\ref{exa:02} for an explicit illustration.

%
%
\end{remark}

\begin{example}\label{exa:01}
Let $\mathbb{R}^{2}$ be endowed with a para-K\"{a}hler--Norden structure $(\varphi, g)$ defined in polar coordinates by
$$g= dr^{2}+r^{2}d\theta^{2},\quad  
\varphi=\left( \begin{array}{ccc}
\sin 2\theta  & r\cos 2\theta\\
\dfrac{1}{r} \cos 2\theta &-\sin 2\theta  
\end{array} \right).$$

Let $\xi = f(r)e_{1}$ be a vector field such that $f$ is a smooth nonzero real-valued function depending only on the variable $r$. With respect to the orthonormal frame
\begin{eqnarray*}
e_{1}=\partial_{r}, \quad e_{2}=\dfrac{1}{r}\partial_{\theta}.
\end{eqnarray*}
we have, 
\begin{eqnarray*}
\varphi e_{1}=\sin 2\theta e_{1}+\cos 2\theta e_{2}, \quad \varphi e_{2}=\cos 2\theta e_{1}-\sin 2\theta e_{2}.
\end{eqnarray*}

The Levi-Civita connection satisfies
\begin{eqnarray}\label{eq:21} 
\nabla_{e_{1}}e_{1}= \nabla_{e_{1}}e_{2}=0, \quad \nabla_{e_{2}}e_{1}=\dfrac{1}{r}e_{2}, \quad \nabla_{e_{2}}e_{2}=-\dfrac{1}{r}e_{1},
\end{eqnarray}
and the curvature tensor vanishes identically, that is, $(\mathbb{R}^{2},\varphi, g)$ is a flat para-K\"{a}hler--Norden manifold, it follows that $S(\xi)=0$.

Using equations \eqref{eq:09}, \eqref{eq:11}, and \eqref{eq:21}, through simple computations, we arrive at:
\begin{eqnarray*}
\bar{\Delta}\xi=\dfrac{1}{r^{2}}(-r^{2}f^{\prime\prime}-rf^{\prime}+f)e_{1},
\end{eqnarray*}
and
\begin{eqnarray*}
\bar{\Delta}^{2}\xi=\dfrac{1}{r^{4}}(r^{4} f^{(4)} + 2 r^{3} f^{(3)} - 3 r^{2} f^{\prime\prime} + 3 r f^{\prime} - 3 f)e_{1}.
\end{eqnarray*}

The vector field $\xi = f(r)e_{1}$ is harmonic vector field if and only if  $\bar{\Delta}\xi=0$, which is equivalent to the second-order differential equation 
\begin{eqnarray*}
-r^{2}f^{\prime\prime}-rf^{\prime}+f=0,
\end{eqnarray*}

whose general solution is 
\begin{eqnarray*}
f(r)= \frac{c_{1}}{r}+c_{2}r,
\end{eqnarray*}
where $c_{1}$ and $c_{2}$ are real constants not both equal to zero. Since $S(\xi)=0$, from \eqref{eq:13}, the vector fields 
$$\xi = (\frac{c_{1}}{r}+c_{2}r)e_{1}$$ 
are also harmonic maps.

From Theorem~\ref{thm:05}, we find that $\xi = f(r)e_{1}$ is a biharmonic vector field if and only if $\bar{\Delta}^{2} \xi = 0$, which is equivalent to the fourth-order differential equation
\begin{equation*}
r^{4} f^{(4)} + 2 r^{3} f^{(3)} - 3 r^{2} f^{\prime\prime} + 3 r f^{\prime} - 3 f = 0,
\end{equation*}

whose general solution is
\begin{eqnarray*}
f(r)= \frac{c_{1}}{r}+c_{2} r+c_{3} r\ln r+c_{4} r^{3},
\end{eqnarray*}
where $c_{1}, c_{2}, c_{3}$ and $c_{4}$ are real constants, not all zero.

By Theorem~\ref{thm:05}, the vector fields
$$\xi = \left(\frac{c_{1}}{r} + c_{2} r + c_{3} r \ln r + c_{4} r^{3}\right) e_{1}$$
are also biharmonic maps.

Note that the vector fields
$$\xi_{1} = \left(\frac{c_{1}}{r} + c_{2} r\right) e_{1}$$
are both harmonic and biharmonic vector fields, and consequently harmonic and biharmonic maps.

On the one hand, the vector fields
$$\xi_{2} = \left(c_{3} r \ln r + c_{4} r^{3}\right) e_{1}$$
are biharmonic but not harmonic vector fields, that is, they are proper biharmonic vector fields; moreover, they define proper biharmonic maps.

On the other hand, the vector fields
$$\xi = \left(\frac{c_{1}}{r} + c_{2} r + c_{3} r \ln r + c_{4} r^{3}\right) e_{1}$$
are biharmonic but non-parallel, since
$$\nabla_{e_{1}} \xi =
\left(-\frac{c_{1}}{r^{2}} + c_{2} + c_{3}(\ln r + 1) + 3 c_{4} r^{2}\right) e_{1}\neq 0.$$

In contrast, the vector field
$$\xi_{3} = \sin\theta\, e_{1} + \cos\theta\, e_{2}$$
is parallel, since $\nabla_{e_{1}} \xi_{3} = \nabla_{e_{2}} \xi_{3} = 0$, and hence is biharmonic. 
\end{example}

\begin{example}\label{exa:02}
Given a para-K\"{a}hler-Norden manifold $\left(\mathbb{R}^{4},g,\varphi\right)$ such that
$$g= dx^{2}+e^{-2x}dy^{2}+dz^{2}+e^{-2z}dt^{2},\quad \varphi=diag(1, 1, -1, -1).$$
With respect to the orthonormal frame
\begin{eqnarray*}
e_{1}=\partial_{x}, \; e_{2}=e^{x}\partial_{y}, \; e_{3}=\partial_{z}, \; e_{4}=e^{z}\partial_{t}.
\end{eqnarray*}
we have, 
\begin{eqnarray}\label{eq:22} 
\varphi e_{1}=e_{1}, \; \varphi e_{2}=e_{2}, \; \varphi e_{3}=-e_{3}, \  \varphi e_{4}=-e_{4}.
\end{eqnarray} 
The nonzero covariant derivatives \(\nabla_{e_i} e_j\), \(i,j = 1,\dots,4\), are given by
\begin{eqnarray}\label{eq:23}
\nabla_{e_{2}}e_{1}=-e_{2}, \; \nabla_{e_{2}}e_{2}=e_{1},\; \nabla_{e_{4}}e_{3}=-e_{4},\; \nabla_{e_{4}}e_{4}=e_{3}.
\end{eqnarray} 
The nonzero curvature terms \(\mathrm{R}(e_i, e_j)e_k\), $i,j,k = 1,\dots,4$, are
\begin{align}\label{eq:24}
\mathrm{R}(e_{1}, e_{2})e_{1}=e_{2},\;& \mathrm{R}(e_{1}, e_{2})e_{2}=-e_{1},\; \mathrm{R}(e_{3}, e_{4})e_{3}=e_{4}\nonumber\\
 &\mathrm{R}(e_{3}, e_{4})e_{4}=-e_{3}.
\end{align}

Let's assume the vector field $\xi = f(x)\,e_{1}$, such that $f$ is a smooth, nonzero real-valued depending only on the variable $x$. Using equations~\eqref{eq:09}, \eqref{eq:10}, \eqref{eq:22}, \eqref{eq:23}, and \eqref{eq:24} through simple computations, shows that
\begin{eqnarray}\label{eq:25}
\bar{\Delta} e_{1}= S(e_{1}) =e_{1}.
\end{eqnarray}
Combining formulas~\eqref{eq:11},~\eqref{eq:12} and~\eqref{eq:25}, we obtain
\begin{eqnarray*}
S(\xi)=f^{2}e_{1},
\end{eqnarray*}
\begin{eqnarray*}
\bar{\Delta}\xi=(-f^{\prime\prime}+f^{\prime}+f)e_{1},
\end{eqnarray*}
and
\begin{eqnarray*}
\bar{\Delta}^{2}\xi=(f^{(4)}-2f^{(3)}-f^{\prime\prime}+2f^{\prime}+f)e_{1}.
\end{eqnarray*}
Moreover, the nonzero terms $\nabla_{e_i}\xi$ and $\nabla_{e_i}S(\xi)$, $i=1,\dots,4$, are given by
\begin{eqnarray*}
\nabla_{e_{1}}\xi=f^{\prime}e_{1}, \; \nabla_{e_{2}}\xi=-fe_{2},\; \nabla_{e_{1}}S(\xi)=2f^{\prime}fe_{1},\; \nabla_{e_{2}}S(\xi)=-f^{2}e_{2}.
\end{eqnarray*}
Using these expressions, we obtain
\begin{eqnarray*}
\mathrm{Tr}_{g}\big((\nabla_{\ast}R)(\ast,S(\xi))\xi+ \mathrm{R}(\ast,\nabla_{\ast}S(\xi))\xi+2\mathrm{R}(\ast,S(\xi))\nabla_{\ast}\xi\big)=2f^{3}e_{1}.
\end{eqnarray*}

The vector field $\xi = f(x)e_{1}$ is harmonic vector field if and only if  $\bar{\Delta}\xi=0$, which is equivalent to 
\begin{eqnarray*}
-f^{\prime\prime}+f^{\prime}+f=0,
\end{eqnarray*}
whose general solution is 
\begin{eqnarray*}
f(x)= c_{1} e^{\frac{1-\sqrt{5}}{2}x}+c_{2} e^{\frac{1+\sqrt{5}}{2}x},
\end{eqnarray*}
where $c_{1},c_{2}\in\mathbb{R}$ are not both zero. Since $S(\xi)\neq0$, it follows from \eqref{eq:13} that the vector fields $$\xi = (c_{1} e^{\frac{1-\sqrt{5}}{2}x}+c_{2} e^{\frac{1+\sqrt{5}}{2}x})e_{1}$$ 
are not harmonic maps.

Using Corollary \ref{cor:01}, we find that $\xi = f(x)e_{1}$ is biharmonic vector field if and only if $f$ satisfies the following fourth-order nonlinear differential equation: 
\begin{eqnarray}\label{eq:26}
f^{(4)}-2f^{(3)}-f^{\prime\prime}+2f^{\prime}+f=2f^{3}.
\end{eqnarray} 
By the Cauchy–Lipschitz (Picard--Lindel\"of) theorem, for any initial data 
$$(f(x_{0}), f^{\prime}(x_{0}), f^{\prime\prime}(x_{0}), f^{\prime\prime\prime}(x_{0}))\in \mathbb{R}^{4},$$
there exists a unique local solution of~\eqref{eq:26} defined in a neighborhood of any point $x_{0}$.
Furthermore,  one can explicitly construct global bounded solutions. For instance, using direct computations (e.g., with \textit{Mathematica}), we find the solutions
$$f(x)=\pm\sqrt{\frac{3}{8}}\,\text{sech}\left(\frac{x}{\sqrt2}\right).$$
Consequently, the vector fields
$$\xi_{1}=\pm\sqrt{\frac{3}{8}}\,\text{sech}\left(\frac{x}{\sqrt2}\right)e_{1},$$
are biharmonic but not harmonic vector fields, they are, proper biharmonic vector fields.

It is clear that the vector fields 
$$\xi=(c_{1} e^{\frac{1-\sqrt{5}}{2}x}+c_{2} e^{\frac{1+\sqrt{5}}{2}x})e_{1}$$ 
are harmonic vector field but non biharmonic vector field.
\end{example}

\section{Interpolating sesqui-harmonicity of vector field $\xi:(M^{2m},\varphi,g)\to (TM,g^{\varphi})$}\label{sec:05}
\begin{theorem}
Given a para-K\"{a}hler--Norden manifold $(M^{2m}, \varphi, g)$ and its tangent bundle $(TM,g^{\varphi})$ endowed with the $\varphi$-Sasaki metric. For a vector field $\xi$ on $M$. Then the interpolating sesqui-tension field $\tau_{\delta_{1},\delta_{2}}(\xi)$ associated with $X$ is given by
\begin{eqnarray}\label{eq:27}
\tau_{\delta_{1},\delta_{2}}(\xi)&=&{}^{H}\!\big[\delta_{1} S(\xi)+\delta_{2}\bar{\Delta}S(\xi)+\delta_{2}\mathrm{R}(\varphi \xi,\bar{\Delta}\xi)S(\xi)-\delta_{2}\mathrm{Tr}_{g}\big(\mathrm{R}(S(\xi),\ast)\ast\nonumber\\
&&+(\nabla_{S(\xi)}R)(\varphi \xi,\nabla_{\ast}\xi)\ast+\mathrm{R}(\varphi \xi,\nabla_{\ast}\xi)\nabla_{\ast}S(\xi)\nonumber\\
&&-\mathrm{R}(\varphi \xi,\mathrm{R}(\ast,S(\xi)) \xi)\ast-\mathrm{R}(\varphi \xi,\nabla_{\ast}\bar{\Delta}\xi)\ast-\mathrm{R}(\varphi\bar{\Delta}\xi,\nabla_{\ast}\xi)\ast\big)\big]\nonumber\\
&&+{}^{V}\!\big[-\delta_{1}\bar{\Delta}\xi-\delta_{2}\bar{\Delta}^{2}\xi+\delta_{2}\mathrm{Tr}_{g}\big((\nabla_{\ast}R)(\ast,S(\xi))\xi\nonumber\\
&&+ \mathrm{R}(\ast,\nabla_{\ast}S(\xi))\xi+2\mathrm{R}(\ast,S(\xi))\nabla_{\ast}\xi\big)\big].  
\end{eqnarray}
\end{theorem}

\begin{proof}
Using formulas \eqref{eq:05}, \eqref{eq:08}, and \eqref{eq:14}, a direct computation yields the expression given in \eqref{eq:27}.
\end{proof}

\begin{theorem}\label{thm:07}
Given a para-K\"{a}hler--Norden manifold $(M^{2m}, \varphi, g)$ and its tangent bundle $(TM,g^{\varphi})$ endowed with the $\varphi$-Sasaki metric. For a vector field $\xi$ on $M$. Then $\xi$ is an interpolating sesqui-harmonic map if and only if the following conditions are satisfied: 
\begin{align*}
\delta_{1} S(\xi)&+\delta_{2}\bar{\Delta}S(\xi)+\delta_{2}\mathrm{R}(\varphi \xi,\bar{\Delta}\xi)S(\xi)-\delta_{2}\mathrm{Tr}_{g}\big(\mathrm{R}(S(\xi),\ast)\ast\nonumber\\
&+(\nabla_{S(\xi)}R)(\varphi \xi,\nabla_{\ast}\xi)\ast+\mathrm{R}(\varphi \xi,\nabla_{\ast}\xi)\nabla_{\ast}S(\xi)\nonumber\\
&-\mathrm{R}(\varphi \xi,\mathrm{R}(\ast,S(\xi)) \xi)\ast-\mathrm{R}(\varphi \xi,\nabla_{\ast}\bar{\Delta}\xi)\ast\nonumber\\
&-\mathrm{R}(\varphi\bar{\Delta}\xi,\nabla_{\ast}\xi)\ast\big)=0,
\end{align*}
and
\begin{align*}
\delta_{1}\bar{\Delta}\xi&+\delta_{2}\bar{\Delta}^{2}\xi-\delta_{2}\mathrm{Tr}_{g}\big((\nabla_{\ast}R)(\ast,S(\xi))\xi+ \mathrm{R}(\ast,\nabla_{\ast}S(\xi))\xi\nonumber\\
&+2\mathrm{R}(\ast,S(\xi))\nabla_{\ast}\xi\big)=0. 
\end{align*}
\end{theorem}

As an immediate consequence of Theorem~\ref{thm:07}, we obtain the following result.

\begin{corollary}
Given a para-K\"{a}hler--Norden manifold $(M^{2m}, \varphi, g)$ and its tangent bundle $(TM,g^{\varphi})$ endowed with the $\varphi$-Sasaki metric. Then every parallel vector field on $M$ is an interpolating sesqui-harmonic map.
\end{corollary}

\begin{definition}
Given a compact para-K\"{a}hler--Norden manifold $(M^{2m},\varphi,g)$ and its tangent bundle $(TM,g^{\varphi})$ endowed with the $\varphi$-Sasaki metric. For a vector field $\xi$ on $M$, the interpolating sesqui-energy functional $\mathrm{E}_{\delta_{1},\delta_{2}}$ associated with $\xi$ is defined as the interpolating sesqui-energy of the corresponding map $\xi:(M^{2m},\varphi,g) \to (TM,g^{\varphi})$.
\end{definition}

\begin{definition}
Given a para-K\"{a}hler--Norden manifold $(M^{2m}, \varphi, g)$ and its tangent bundle $(TM,g^{\varphi})$ endowed with the $\varphi$-Sasaki metric. A vector field $\xi$ on $M$ is said to be an interpolating sesqui-harmonic vector field if the associated map $\xi \colon (M^{2m},\varphi,g)  \to  (TM,g^{\varphi})$ is a critical point of the interpolating sesqui-energy functional $\mathrm{E}_{\delta_{1},\delta_{2}}$, where the variations are restricted to maps induced by vector fields on $M$.
\end{definition}

In the following theorem, we compute the first variation of the interpolating sesqui-energy functional restricted to the space of vector fields.

\begin{theorem}\label{thm:08}
Given a compact oriented para-K\"{a}hler--Norden manifold $(M^{2m}, \varphi, g)$ and its tangent bundle $(TM,g^{\varphi})$ endowed with the $\varphi$-Sasaki metric. Let $\xi$ be a vector field on $M$, and $\mathrm{E}_{\delta_{1},\delta_{2}}$ be the interpolating sesqui-energy functional restricted to the space of vector fields on $M$. Then, for any smooth $1$-parameter variation $\Phi:M\times(-\varepsilon,\varepsilon) \to  TM$
of $\xi$ through vector fields, the following formula holds:
\begin{eqnarray*}
\frac{d}{dt}\mathrm{E}_{\delta_{1},\delta_{2}}(\xi_{t})\Big|_{t=0}
&=&2\int_{M} g\big(
\delta_{1}\bar{\Delta}\xi+\delta_{2}\bar{\Delta}^{2}\xi
-\delta_{2}\mathrm{Tr}_{g}\big(
(\nabla_{\ast}R)(\ast,S(\xi))\xi \nonumber\\
&&+ \mathrm{R}(\ast,\nabla_{\ast}S(\xi))\xi+2\mathrm{R}(\ast,S(\xi))\nabla_{\ast}\xi
\big),\varphi V\big)\,v_{g},
\end{eqnarray*}
where $\Phi(p,t)=\xi_{t}(p)\in T_{p}M$ for all $p\in M$ and $|t|<\varepsilon$ $(\varepsilon>0)$. Moreover, the vector field $W$ on $M$ is defined by
\begin{eqnarray*}
W(p)= \lim_{t\to 0}\frac{1}{t}\big(\xi_{t}(p)-\xi(p)\big)= \frac{d}{dt}\Phi_{p}(t)\Big|_{t=0},\; p\in M,
\end{eqnarray*}
where $\Phi_{p}(t)=\xi_{t}(p)$ for $(p,t)\in M\times(-\varepsilon,\varepsilon)$.
\end{theorem}

\begin{proof}
The proof follows the same strategy as that of Theorem~$\ref{thm:04}$. We consider the smooth $1$-parameter variation $\Phi:M\times(-\epsilon,\epsilon) \to TM$ of $\xi$ through vector fields, namely, $\Phi(p,t)=\xi_{t}(p)\in T_{p}M,\; \Phi(p,0)=\xi(p)$,
for all $(p,t)\in M\times(-\varepsilon,\varepsilon)$.  
From \eqref{eq:04}, we have
\begin{eqnarray*}
\mathrm{E}_{\delta_{1},\delta_{2}}(\xi_{t})=2\delta_{1}E(\xi_{t})+2\delta_{2}\mathrm{E}_{2}(\xi_{t}).
\end{eqnarray*}

By the well-known first variation formulas in the theory of harmonic and biharmonic maps (see, for example, \cite{E.S,Jia}), it follows that
\begin{eqnarray}\label{eq:28}
\frac{d}{dt}\mathrm{E}_{\delta_{1},\delta_{2}}(\xi_{t})\Big|_{t=0}
&=&2\delta_{1}\frac{d}{dt}E(\xi_{t})\Big|_{t=0}
+2\delta_{2}\frac{d}{dt}\mathrm{E}_{2}(\xi_{t})\Big|_{t=0}\nonumber\\
&=&-2\delta_{1}\int_{M} g^{\varphi}(\tau(\xi),\mathcal{V})\,v_{g}
-2\delta_{2}\int_{M} g^{\varphi}(\tau_{2}(\xi),\mathcal{V})\,v_{g}\nonumber\\
&=&-2\int_{M} g^{\varphi}(\tau_{\delta_{1},\delta_{2}}(\xi),\mathcal{V})\,v_{g},
\end{eqnarray}
where $\mathcal{V}$ denotes the infinitesimal variation induced by $\Phi$, given by
\begin{eqnarray*}
\mathcal{V}(p)=d_{(p,0)}\Phi\!\left(0,\frac{d}{dt}\right)\Big|_{t=0}
=d\Phi_{p}\!\left(\frac{d}{dt}\right)\Big|_{t=0}=\frac{d}{dt}\xi_{t}(p)\Big|_{t=0}
\in T_{(p,\xi(p))}TM.
\end{eqnarray*}

Finally, taking into account \eqref{eq:20}, \eqref{eq:27}, and \eqref{eq:28}, we obtain
\begin{eqnarray*}
\frac{d}{dt}\mathrm{E}_{\delta_{1},\delta_{2}}(\xi_{t})\Big|_{t=0}
&=&-2\int_{M} g^{\varphi}(\tau_{\delta_{1},\delta_{2}}(\xi),{}^{V}\!W)\,v_{g}\\
&=&2\int_{M} g\big(\delta_{1}\bar{\Delta}\xi
+\delta_{2}\bar{\Delta}^{2}\xi-\delta_{2}\mathrm{Tr}_{g}\big((\nabla_{\ast}R)(\ast,S(\xi))\xi \nonumber\\
&&+\mathrm{R}(\ast,\nabla_{\ast}S(\xi))\xi+2\mathrm{R}(\ast,S(\xi))\nabla_{\ast}\xi\big),\varphi W\big)\,v_{g},
\end{eqnarray*}
which completes the proof.
\end{proof}

\begin{remark}\label{rem:2}
As in Theorem~\ref{thm:04}, concerning the first variation of the bienergy, Theorem~\ref{thm:08} remains valid when $(M^{2m},\varphi,g)$ is a non-compact para-K\"{a}hler--Norden manifold. Indeed, if $M$ is non-compact, one can choose an open subset $D \subset M$ with compact closure and consider an arbitrary variation vector field $W$ whose support is contained in $D$. In this case, Theorem~\ref{thm:08} holds in the following form:
\begin{eqnarray*}
\frac{d}{dt}\mathrm{E}_{\delta_{1},\delta_{2}}(\xi_{t})\Big|_{t=0}
&=& 2\int_{D} g\big(\delta_{1}\bar{\Delta}\xi+\delta_{2}\bar{\Delta}^{2}\xi
-\delta_{2}\mathrm{Tr}_{g}\big((\nabla_{\ast}R)(\ast,S(\xi))\xi \nonumber\\
&& + \mathrm{R}(\ast,\nabla_{\ast}S(\xi))\xi
+2\mathrm{R}(\ast,S(\xi))\nabla_{\ast}\xi\big),\varphi W\big)\,v_{g}.
\end{eqnarray*}
\end{remark}

\begin{corollary}\label{cor:03}
Given a para-K\"{a}hler--Norden manifold $(M^{2m}, \varphi, g)$ and its tangent bundle $(TM,g^{\varphi})$ endowed with the $\varphi$-Sasaki metric. A vector field $\xi$ on $M$ is an interpolating sesqui-harmonic vector field if and only if
\begin{align*}
\delta_{1}\bar{\Delta}\xi&+\delta_{2}\bar{\Delta}^{2}\xi-\delta_{2}\mathrm{Tr}_{g}\big((\nabla_{\ast}R)(\ast,S(\xi))\xi+ \mathrm{R}(\ast,\nabla_{\ast}S(\xi))\xi\nonumber\\
&+2\mathrm{R}(\ast,S(\xi))\nabla_{\ast}\xi\big)=0. 
\end{align*}
\end{corollary}

From Theorem~\ref{thm:07} and Corollary~\ref{cor:03}, we obtain the following results.

\begin{corollary}
Given a para-K\"{a}hler--Norden manifold $(M^{2m}, \varphi, g)$ and its tangent bundle $(TM,g^{\varphi})$ endowed with the $\varphi$-Sasaki metric. A vector field $\xi$ on $M$ is an interpolating sesqui-harmonic map if and only if $\xi$ is an interpolating sesqui-harmonic covector field and satisfies
\begin{align*}
\delta_{1} S(\xi)&+\delta_{2}\bar{\Delta}S(\xi)+\delta_{2}\mathrm{R}(\varphi \xi,\bar{\Delta}\xi)S(\xi)-\delta_{2}\mathrm{Tr}_{g}\big(\mathrm{R}(S(\xi),\ast)\ast\nonumber\\
&+(\nabla_{S(\xi)}R)(\varphi \xi,\nabla_{\ast}\xi)\ast+\mathrm{R}(\varphi \xi,\nabla_{\ast}\xi)\nabla_{\ast}S(\xi)\nonumber\\
&-\mathrm{R}(\varphi \xi,\mathrm{R}(\ast,S(\xi)) \xi)\ast-\mathrm{R}(\varphi \xi,\nabla_{\ast}\bar{\Delta}\xi)\ast\nonumber\\
&-\mathrm{R}(\varphi\bar{\Delta}\xi,\nabla_{\ast}\xi)\ast\big)=0,
\end{align*}
\end{corollary}

\begin{theorem}\label{thm:09}
Given a flat para-K\"{a}hler--Norden manifold $(M^{2m}, \varphi, g)$ and its tangent bundle $(TM,g^{\varphi})$ endowed with the $\varphi$-Sasaki metric. For a vector field $\xi$ on $M$, the following statements are equivalent:
\item $(i)$ $\xi$ is interpolating sesqui-harmonic map on $M$.
\item $(ii)$ $\xi$ is interpolating sesqui-harmonic vector field on $M$.
\item $(iii)$ $\delta_{1}\bar{\Delta}\xi+\delta_{2}\bar{\Delta}^{2}\xi=0$.
\end{theorem}

\begin{corollary}
Given a para-K\"{a}hler--Norden manifold $(M^{2m}, \varphi, g)$ and its tangent bundle $(TM,g^{\varphi})$ endowed with the $\varphi$-Sasaki metric. Then every parallel vector field on $M$ is an interpolating sesqui-harmonic vector field.
\end{corollary}

\begin{theorem}
Given a compact oriented para-K\"{a}hler--Norden manifold $(M^{2m}, \varphi, g)$ and its tangent bundle $(TM,g^{\varphi})$ endowed with the $\varphi$-Sasaki metric. Then every harmonic vector field on $M$ is an interpolating sesqui-harmonic vector field.
\end{theorem}

\begin{proof}
The proof follows directly from Theorem~\ref{thm:06}.
\end{proof}

\begin{remark}
In general, if $(M^{2m},\varphi,g)$ is a non-compact para-K\"{a}hler--Norden manifold, there may exist harmonic vector fields that are not interpolating sesqui-harmonic, and vice versa. See Example~\ref{exa:03},~\ref{exa:04} for an explicit illustration.
\end{remark}

\begin{example}\label{exa:03}
We now return to Example~\ref{exa:02} and retain the same notation and computations as above. From Corollary~\ref{cor:03}, we find that the vector field $\xi = f(x)e_{1}$ is an interpolating sesqui-harmonic vector field if and only if the function $f$ satisfies the following fourth-order nonlinear differential equation:
\begin{eqnarray}\label{eq:29}
\qquad\qquad\delta_{2}f^{(4)} -2\delta_{2}f^{(3)} - (\delta_{1}+\delta_{2})f^{\prime\prime} + (\delta_{1}+2\delta_{2})f^{\prime}+ (\delta_{1}+\delta_{2})f= 2\delta_{2}f^{3}.
\end{eqnarray} 

Assume that $\delta_{1}\neq0$ and $\delta_{2}\neq0$. By the Cauchy--Lipschitz (Picard--Lindel\"of) theorem, equation~\eqref{eq:29} admits a unique local solution for any prescribed initial data. In particular, it admits a nontrivial smooth solution $f\in C^{\infty}(\mathbb{R})$ satisfying the localization condition
$$\lim_{|x|\to\infty}f(x)=0.$$

Moreover, if $\delta_{1}\neq0$ and $\dfrac{\delta{1}+\delta{2}}{\delta{2}}>0$, then equation~\eqref{eq:29} admits nonzero constant solutions given by
$$f(x)=\pm\sqrt{\dfrac{\delta{1}+\delta{2}}{2\delta{2}}}.$$
Consequently, the constant vector fields
$$\xi=\pm\sqrt{\dfrac{\delta{1}+\delta{2}}{2\delta{2}}}e_{1},$$
are interpolating sesqui-harmonic vector fields. Since they are neither harmonic nor biharmonic, they provide examples of proper interpolating sesqui-harmonic vector fields.

On the other hand, it is readily verified that the vector fields 
$$\xi=(c_{1} e^{\frac{1-\sqrt{5}}{2}}+c_{2} e^{\frac{1+\sqrt{5}}{2}})e_{1}$$ 
are harmonic vector field but not interpolating sesqui-harmonic vector fields.
\end{example}

\begin{example}\label{exa:04}
Let $\mathbb{R}^{2}$ be endowed with a para-K\"{a}hler--Norden structure $(\varphi, g)$ defined by
$$g= e^{2x}dx^{2}+e^{2y}dy^{2},\quad  
\varphi=\left( \begin{array}{ccc}
0  & e^{y-x}\\
e^{x-y} &0 
\end{array} \right).$$

Consider the vector field $\xi = f(x)e_{1}$, where $f$ is a smooth nonzero real-valued function depending only on the variable $x$. With respect to the orthonormal frame
\begin{eqnarray*}
e_{1}=e^{-x}\partial_{x}, \quad e_{2}=e^{-y}\partial_{y}.
\end{eqnarray*}
we have, 
\begin{eqnarray*}
\begin {aligned}		
&\varphi e_{1}= e_{2}, \quad \varphi e_{2}=e_{1}\\
&\nabla_{e_{i}}e_{j}=0,\; i,j=1,2\\
&\mathrm{R}(e_{i}, e_{j})e_{k}=0,\; i,j,k=1,2.
\end{aligned}
\end{eqnarray*}

Hence $(\mathbb{R}^{2},\varphi, g)$ is a flat para-K\"{a}hler--Norden manifold, and therefore $S(\xi)=0$.

By direct computations using \eqref{eq:09}, \eqref{eq:11}, and \eqref{eq:21}, we obtain
\begin{eqnarray*}
\bar{\Delta}\xi=-e^{-2x}(f^{\prime\prime}-f^{\prime})e_{1},
\end{eqnarray*}
and
\begin{eqnarray*}
\bar{\Delta}^{2}\xi=e^{-4x}(f^{(4)} - 6 f^{(3)} +11 f^{\prime\prime} -6 f^{\prime})e_{1}.
\end{eqnarray*}

The vector field $\xi = f(x)e_{1}$ is harmonic vector field if and only if  $\bar{\Delta}\xi=0$, which is equivalent to the second-order differential equation 
\begin{eqnarray*}
f^{\prime\prime}-f^{\prime}=0.
\end{eqnarray*}
Its general solution is 
\begin{eqnarray*}
f(x)= c_{1}+c_{2}e^{x},
\end{eqnarray*}
where $c_{1}$ and $c_{2}$ are real constants not both zero. Since $S(\xi)=0$, it follows from \eqref{eq:13} that the vector fields
$$\xi = (c_{1}+c_{2}e^{x})e_{1}$$ 
are also harmonic maps.

From Theorem~\ref{thm:05}, we conclude that $\xi = f(x)e_{1}$ is a biharmonic vector field if and only if $\bar{\Delta}^{2} \xi = 0$, which is equivalent to the fourth-order differential equation
\begin{equation*}
f^{(4)} - 6 f^{(3)} +11 f^{\prime\prime} -6 f^{\prime}= 0,
\end{equation*}
The general solution of this equation is
\begin{eqnarray*}
f(x)= c_{1}+c_{2}e^{x}+c_{3} e^{2x}+c_{4} e^{3x},
\end{eqnarray*}
where $c_{1}, c_{2}, c_{3}$ and $c_{4}$ are real constants, not all zero. By Theorem~\ref{thm:05}, the vector fields
$$\xi = \left(c_{1}+c_{2}e^{x}+c_{3} e^{2x}+c_{4} e^{3x}\right) e_{1}$$
are also biharmonic maps.

From Theorem~\ref{thm:09}, we conclude that $\xi = f(x)e_{1}$ is is an interpolating sesqui-harmonic vector field if and only if $$\delta_{1}\bar{\Delta}\xi+\delta_{2}\bar{\Delta}^{2}\xi=0,$$ 
which is equivalent to the fourth-order differential equation
\begin{equation*}
\delta_{2}f^{(4)} - 6 \delta_{2}f^{(3)} + (11\delta_{2}-\delta_{1}e^{2x})f^{\prime\prime} + (-6\delta_{2}+\delta_{1}e^{2x})f^{\prime} = 0.
\end{equation*}

If $\delta_{2}\neq0$ and $\lambda=\dfrac{\delta_{1}}{\delta_{2}}>0$, then the general solution is
\begin{eqnarray}\label{eq:30}
f(x)= c_{1}+c_{2}e^{x}+c_{3} e^{\sqrt{\lambda}e^{x}}+c_{4} e^{-\sqrt{\lambda}e^{x}},
\end{eqnarray}
where $c_{1}, c_{2}, c_{3}$ and $c_{4}$ are real constants, not all zero.

If $\delta_{2}\neq0$ and $\lambda=\dfrac{\delta_{1}}{\delta_{2}}<0$, then the general solution is
\begin{eqnarray}\label{eq:31}
f(x)= c_{1}+c_{2}e^{x}+c_{3} \cos(\sqrt{-\lambda}e^{x})+c_{4} \sin(\sqrt{-\lambda}e^{x}),
\end{eqnarray}
where $c_{1}, c_{2}, c_{3}$ and $c_{4}$ are real constants, not all zero.

By Theorem~\ref{thm:09}, the vector fields  $\xi = f(x) e_{1}$, where $f(x)$ is given by \eqref{eq:30} or \eqref{eq:31}, are also interpolating sesqui-harmonic maps.
\end{example}

\end{document}